\documentclass[10pt,a4paper]{article}

\usepackage[latin1]{inputenc}
\usepackage{amsfonts}
\usepackage[dvipsnames]{xcolor}
\usepackage{amsmath}
\usepackage{amsthm}
\usepackage{amscd}
\usepackage{amssymb}
\usepackage{nicefrac}
\usepackage{mathrsfs}
\usepackage{bbm}
\usepackage{bm}
\usepackage{stackrel}
\usepackage{dsfont}
\usepackage{url}
\usepackage{bbm}
\usepackage{dsfont}
\usepackage{pdfpages}
\usepackage[margin=2.5cm]{geometry}
\usepackage{subfigure}

\numberwithin{equation}{section}

\usepackage
[pagebackref,colorlinks,citecolor=Plum,urlcolor=Periwinkle,linkcolor=DarkOrchid]
{hyperref}

\usepackage[displaymath, mathlines]{lineno}

\newtheorem{theorem}{Theorem}[section]
\newtheorem{proposition}[theorem]{Proposition}

\newtheorem{dfn}[theorem]{Definition}
\newtheorem{lemma}[theorem]{Lemma}
\newtheorem{remark}[theorem]{Remark}

\newcommand{\RR}{\mathbb{R}}

\newcommand{\NN}{\mathbb{N}}

\newcommand{\PP}{\mathbb{P}}
\newcommand{\EE}{\mathbb{E}}

\newcommand{\bs}{\boldsymbol}

\newcommand{\cec}{}

\DeclareMathOperator{\Var}{Var}

\DeclareMathOperator{\inte}{int}

\def\ba#1\ea{\begin{linenomath}\begin{align*}#1\end{align*}\end{linenomath}}
\def\ban#1\ean{\begin{linenomath}\begin{align}#1\end{align}\end{linenomath}}

\begin{document}

\title{Large deviations principle for a stochastic process with random reinforced relocations}
\author{Erion-Stelios Boci\thanks{Department of Mathematical Sciences, University of Bath, Claverton Down, BA2 7AY Bath, UK.\newline Email: e.boci/c.mailler@bath.ac.uk} \and C\'ecile Mailler$^*$\thanks{CM is grateful to EPSRC for support through the fellowship EP/R022186/1.}}
\date{}
\maketitle

\begin{abstract}
Stochastic processes with random reinforced relocations have been introduced 
in a series of papers by Boyer and co-authors (Boyer and Solis Salas 2014, Boyer and Pineda 2016, Boyer, Evans and Majumdar 2017) to model animal foraging behaviour. 
Such a process evolves as a Markov process, except at random relocation times, 
when it chooses a time at random in its whole past according to some ``memory kernel'', 
and jumps to its value at that random time.

We prove a quenched large deviations principle for the value of the process at large times.
The difficulty in proving this result comes from the fact that the 
process is not Markov because of the relocations. Furthermore,
the random inter-relocation times act as a random environment.
\end{abstract}

\section{Introduction}

Consider a discrete-time process on $\mathbb Z^d$ that evolves as a simple random walk, except at some random ``relocation'' times when it chooses a time uniformly at random in its past, and jumps to the site where it was at that time. The waiting times between relocation times, which we call ``run-lengths'' are i.i.d.\ geometric random variable of a given parameter $q\in(0,1)$. This process was first introduced by Boyer and Solis-Salas \cite{BSS14} as a model for animal foraging behaviour; they argue in particular that this model applies to monkeys foraging in their natural environment.

Boyer and Solis-Salas~\cite{BSS14} showed that, when the increments of the simple random walk between relocation times have finite second moment, the process satisfies a central limit theorem with variance~$\log n$, which they called ``anomalous''. Boyer and Pineda~\cite{BP16} generalised this result to the case when the simple random walk increments are heavy-tailed, and Boyer, Evans and Majumdar~\cite{BEM17} allowed the memory to be biased, i.e.\ at relocation times, the walker is more likely to jump back to a site it visited recently than to a site it visited a long time ago.

Mailler and Uribe Bravo~\cite{MUB18} unified all these variants of the original ``monkey walk'' 
of Boyer and Solis-Salas~\cite{BSS14} under an over-arching framework, 
in which the simple random walk can be replaced by any ``ergodic'' 
(in a more general sense than usual) discrete- or continuous-time Markov process, 
and where the run-lengths can have any distribution. 
Their result also holds for a large class of ``memory kernels'',
where the memory kernel is a function that models the bias of the memory.
Under an 8-th moment assumption for the run-length distribution, 
and assuming that the underlying Markov process (between relocation times) 
satisfies a central limit theorem, 
they show a central limit theorem for the position of the walker at large times, 
and thus confirm and generalise the results of~\cite{BSS14, BP16, BEM17}. 
The approach introduced by Mailler and Uribe Bravo~\cite{MUB18} is very different from 
the Fourier calculus approach developed in~\cite{BSS14, BP16, BEM17}, and this 
probabilistic approach is the key to allowing for a more general framework.

\medskip
{\bf Our contributions:}
Our main result (Theorem~\ref{ThLarDevMW}) is a large deviation principle for the position of the walker at large times.
We do this for the general framework of~\cite{MUB18}, i.e.\ we allow the underlying process  
to be any Markov process that satisfies some large deviations principle, 
and allow the run-length distribution to be arbitrary as long as it satisfies some (exponential) 
moment assumption. 
We also treat a large class of memory kernels: we cover all cases treated by~\cite{MUB18}.

The difficulty comes from two factors: (a) the process is non Markovian, and (b) the run-length act as a random environment. Our result is quenched in the sense that it holds conditionally 
on the sequence of run-lengths. It holds almost surely for the ``steep'' kernels of~\cite{MUB18}, and in probability for the ``flat'' ones.
The discrepancy between steep and flat memory kernels is due 
to a renewal theory argument that holds only in a weak sense in the flat case.

Our proof relies on applying G\"artner-Ellis' large deviation theorem (see, e.g.\ \cite[Chapter~V2]{H08}); proving that the assumptions of the theorem hold in our context is the bulk of our paper, 
and this relies on precise estimates of sums of random variables (the additional level of randomness due to the random run-length makes this slightly technical). 

\medskip
In the rest of this introduction, we state and discuss our results.

\subsection{Definition of the model} 
We consider time to be either discrete or continuous, 
i.e.\ $\mathcal T=\NN$ or $\mathcal T=[0,\infty)$. 
The ``monkey process'' $X=(X(t))_{t\in \mathcal T}$
is an $\mathbb R^d$-valued stochastic process that depends on three parameters: 
a semi-group $P=(P_t)_{t\in \mathcal T}$ on $\mathbb R^d$, 
a probability distribution $\phi$ on $\mathcal T$ such that $\phi(\{0\})=0$,
called the ``run-length distribution", 
and a function $\mu : \mathcal T\to \RR^+$ called the ``memory kernel".
Let $Z$ be  a Markov process of semi-group $P$ on $\RR^d$. For all $x\in \RR^d$, we denote by $\PP_x$ the law of $Z$ started at $x$, and by $\PP_x^t$ the law of $(Z(s))_{s<t}$ under $\PP_x$.
Let $\bs L=(L_i)_{i\in \NN}$ be a sequence of $\mathcal T$-valued i.i.d.\ random variables with distribution~$\phi$. We let $T_0=0$ and, for all $n\geq 1$, $T_n=\sum_{i=1}^n L_i$. We call these random times the ``relocation times" of the monkey process.

The monkey process $X=(X(t))_{t\in \mathcal T}$ is defined recursively as follows: 
Fix $X(0)\in\mathbb R^d$ and draw $(X(s))_{s<T_1}$ at random according to $\PP_0^{T_1}$. For all $n\geq 1$, given $(X(s))_{s<T_n}$,
\begin{itemize}
\item[(i)] draw a random variable $R_{n+1}$ in $[0,T_n)$ according to the distribution 
\begin{linenomath}\begin{equation}\label{eq: definition of Rn} 
\PP(R_{n+1}\leq x|T_n)=\frac{\int_0^x \mu (u)\; \mathrm{d}u}{\int_0^{T_n} \mu (u)\; \mathrm{d}u},
\end{equation}\end{linenomath}
\item[(ii)] set $V_{n+1}=X(R_{n+1})$ and draw $(X(s))_{T_n\leq s< T_{n+1}}$ according to $\PP_{V_{n+1}}^{L_{n+1}}$.
\end{itemize}
This defines the process $X=(X(t))_{t\in \mathcal T}$, which we call the
monkey process of semigroup $P$, 
run-length distribution $\phi$ and memory kernel $\mu$.

\begin{figure}[ht]
\centering
\includegraphics[width=10cm,page=12]{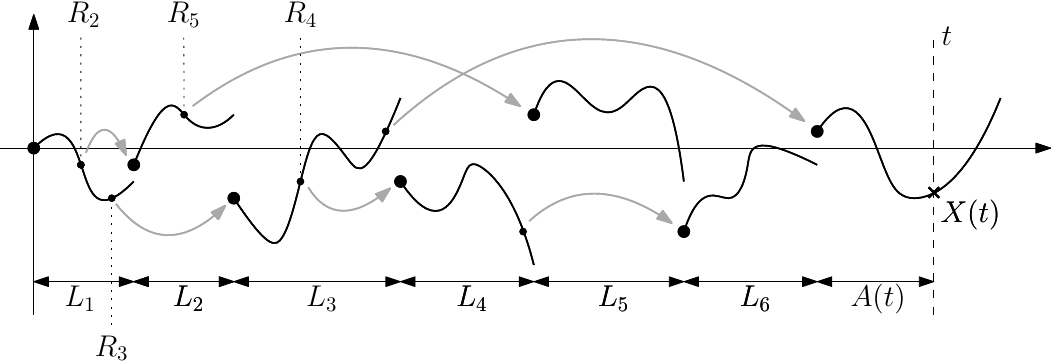}
\caption{A realisation of the monkey walk. The grey arrows point from $X(R_{n+1})$ to $X(T_n)$ for each $n\geq 1$.}\label{fig:def}
\end{figure}

If the memory kernel $\mu$ is uniform, then, at each relocation time, the walker chooses a time uniformly in its past, and relocates to where it was at that random time. This is the original process introduced by Boyer and Solis-Salas \cite{BSS14}. The case of more general memory kernels was introduced by Boyer, Evans and Majumdar~\cite{BEM17} and studied by Mailler and Uribe Bravo~\cite{MUB18}.
As in \cite{MUB18}, we consider the following two families of memory kernels:  for all $x\in\mathcal T$,
\begin{linenomath}\begin{equation}\label{eq: definition of mu1}
\mu_1(x)=\frac{\alpha}{x}(\log x)^{\alpha -1}\mathrm{e}^{\beta (\log x)^\alpha},
\end{equation}\end{linenomath}
where $\alpha>0$ and $\beta\geq 0$,
and 
\begin{linenomath}\begin{equation}\label{eq: definition of mu2}
\mu_2(x)=\gamma \delta x^{\delta -1}\mathrm{e}^{\gamma x^{\delta}},
\end{equation}\end{linenomath}
where $\gamma>0$ and $\delta\in (0,\nicefrac{1}{2}]$.

\begin{remark}
Note that the case $\mu=\mu_1$ with $\alpha=\beta=1$ corresponds to the uniform memory case, 
i.e., for all $n\geq 1$, the random time $R_{n+1}$ is chosen uniformly in the interval $[0, T_n)$. 
Some of the memory kernels (e.g., $\mu_1$ with $\alpha = 1$ and $\beta<1$) are decreasing: 
the walker is more likely to relocate to positions it visited near time~0 
than positions it visited recently. 
Others, such as $\mu_2$, are increasing: 
the walker is more likely to relocate to positions visited more recently.
\end{remark}

\begin{remark}
Note that, when $\beta = 0$, $\mu_1$ is not integrable at zero, which causes problems 
when defining the relocation times as random variable of probability density proportional to $\mu$.
To fix this issue, one can re-define $\mu_1$ as $\mu_1(x) =\frac\alpha{x+1}(\log (x+1))^{\alpha-1}$, for all $x>0$. For simplicity of notation, and since it does not affect the final results, we keep $\mu_1$ defined as it is.
\end{remark}

\begin{remark}
Although the model can also be defined for $\mu=\mu_2$ and $\delta>\nicefrac12$, the central limit theorem of~\cite{MUB18} as well as our main results do not hold in that case, 
which corresponds to a memory kernel that is very ``steep'', 
i.e.\ the walker is likely to relocate to positions it visited in its very recent past.
In a forthcoming paper, we prove a CLT for the position of the walker at large times in the monkey random walk defined with this very steep memory kernel.
\end{remark}

\subsection{Statement of the results}
We prove that, under some assumptions on the parameters of the monkey process $X$,
conditionally on $\bs L = (L_i)_{i\geq 1}$, 
a large deviation principle holds for $X(t)/s(t)$ almost surely, 
where, for all $x>0$,
\begin{linenomath}\begin{equation}\label{eq:st}s(x)=\begin{cases}
(\log x)^{\alpha} &\text{if}\;\mu=\mu_1, \beta\neq 0 \\
\alpha\log\log x &\text{if}\;\mu=\mu_1,\beta=0 \\
\gamma x^{\delta} &\text{if}\;\mu=\mu_2.
\end{cases}\end{equation}\end{linenomath}
Note that $s = \log (\int \mu)$ (up to a constant factor $\beta$ in the case when $\mu = \mu_1$ and $\beta\neq 0$).

Our assumptions are
\begin{itemize}
\item[{\bf (A1)}] One of the following holds: 
\begin{itemize}
\item[{\bf (A1a)}] $\mu =\mu_1$, $\beta\neq 0$ and $\alpha\geq 1$,
or $\mu = \mu_2$ and $\delta\leq \nicefrac12$ (case when $s(x)\geq \log x$ for all $x>0$),
\item[{\bf (A1b)}] $\mu =\mu_1$, $\beta=0$ or $\alpha< 1$ (case when $s(x)<\log x$ for all $x>0$).
\end{itemize}
\item[{\bf (A2)}] There exists a differentiable function $\Lambda_Z : \RR^d\to{\cec [0,\infty]}$ such that, 
for all $\zeta, x\in \RR^d$,
\begin{linenomath}\begin{equation}\label{eq: ZLDPcondition}
\lim_{t\to+\infty}\frac1t \log \EE_x[\mathrm{e}^{\zeta \cdot Z(t)}]=\Lambda_Z(\zeta),
\end{equation}\end{linenomath}
where $\zeta\cdot Z(t)$ is the scalar product of $\zeta$ and $Z(t)\in\mathbb R^d$.
Moreover, for all $t>0$ and $\zeta\in\mathbb R^d$, 
\begin{equation}\label{eq:upper_bound_Z}
\sup_{0\leq u\leq t}\mathbb E_x[\mathrm e^{\zeta\cdot Z(u)}]<+\infty.
\end{equation}
(Recall that $Z$ is the Markov process of semi-group $P$ and started at $x$ on $\mathbb R^d$.)
\item[{\bf (A3)}] {\cec If, for all $\xi\in \mathbb R$,
\begin{linenomath}\begin{equation}\label{eq:def_Lambda}
\Lambda(\xi):=
\frac{\EE[{\mathrm{e}^{\xi L}-1-\xi L}]}{\xi\EE[L]} 
\end{equation}\end{linenomath} 
where $L$ is a $\phi$-distributed random variable,
and if $\mathcal D = \{\zeta\in\mathbb R^d \colon \Lambda \circ \Lambda_Z(\zeta)<\infty\}$,
then, $\Lambda$ is continuous and
\begin{itemize}
\item[(o)] $0 \in \mathrm{Int}(\mathcal D)$ ($\mathrm{Int}(\mathcal D)$ denotes the interior of $\mathcal D$),
\item[(i)] $\Lambda \circ \Lambda_Z$ is lower semi-continuous on $\mathbb R^d$,
\item[(ii)] $\Lambda \circ \Lambda_Z$ is differentiable on $\mathrm{Int}(\mathcal D)$,
\item[(iii)] either $\mathcal D = \mathbb R^d$ or $\lim_{\zeta \to \partial D} |\Lambda \circ \Lambda_Z(\zeta)| = \infty$.
\end{itemize}}
\end{itemize}

\begin{theorem}[Quenched large deviations principle]\label{ThLarDevMW}
Let $X$ be the monkey process of semi-group~$P$, run-length distribution $\phi$ and memory kernel $\mu$.
Under Assumptions {\rm {\bf (A1-3)}}, for all $X(0)\in\mathbb R^d$, 
for all $x>0$,
\begin{linenomath}\begin{equation}\label{eq:Result}
-\inf_{\|y\|>x} \Lambda_X^{\!*}(y)
\leq 
\liminf_{t\to+\infty} \frac{\log\mathbb P(\|X(t)\|> xs(t)\mid \bs L)}{s(t)} 
\leq \limsup_{t\to+\infty} \frac{\log\mathbb P(\|X(t)\|\geq xs(t)\mid \bs L)}{s(t)}
= -\inf_{\|y\|\geq x} \Lambda_X^{\!*}(y),
\end{equation}\end{linenomath}
where $s(t)$ is defined in Equation \eqref{eq:st}, $\Lambda^*_X$ is the Legendre transform of $\Lambda_X=\Lambda\circ \Lambda_Z$, and $\|\!\cdot\!\|$ is the $L^2$-norm on $\mathbb R^d$.
Equation~\eqref{eq:Result} holds almost surely under {\rm\bf(A1a)} and in probability under {\rm\bf(A1b)}.
\end{theorem}

\begin{remark}
Unfortunately, we are not able to prove that
for all $x>0$,
\begin{linenomath}\begin{equation}\label{eq:annealed}
-\inf_{\|y\| >x} \Lambda_X^{\!*}(y)
\leq 
\liminf_{t\to+\infty} \frac{\log\mathbb P(\|X(t)\|> xs(t))}{s(t)} 
\leq \limsup_{t\to+\infty} \frac{\log\mathbb P(\|X(t)\|\geq xs(t))}{s(t)}
= -\inf_{\|y\|\geq x} \Lambda_X^{\!*}(y),
\end{equation}\end{linenomath}
i.e.\ an annealed version of Theorem~\ref{ThLarDevMW}.
The lower bound in~\eqref{eq:annealed} is a straightforward consequence of the lower bound  
in Theorem~\ref{ThLarDevMW} and Jensen's inequality (since $\log$ is concave).
However, the upper bound is not covered by our results, 
and our method does not seem to be applicable to that case. 
Because it also takes into account the fact that run-length 
could become exceptionally long on some rare event, it is not surprising that the annealed 
large deviations principle is more difficult to prove than the quenched one.
{\cec We still believe that Theorem~\ref{ThLarDevMW} would hold in the annealed sense, at least for run-lengths distributions with light enough tails.}
\end{remark}

This result provides a large deviations principle for the monkey process. It complements the results of~\cite{MUB18} who proved that {\it if $Z$ satisfies a central limit theorem and $\phi$ some moment assumptions then $X$ also satisfies a central limit theorem.} Similarly, although the proofs are more involved, as often for large deviations principles, we prove that {\it if $Z$ satisfies a large deviation principle and $\phi$ some moment assumptions then $X$ also satisfies a large deviations principle.}

In the case of the uniform memory kernel, we get that the large deviations decay as a negative power of time. This behaviour is quite unusual since one would typically get large deviations that decay exponentially in time. This behaviour is not surprising though, since, by~\cite{MUB18}, the monkey process with uniform memory kernel diffuses at logarithmic speed (under the assumption that $Z$ diffuses at linear speed).
{\cec Although polynomially-decaying large deviations are not so common in the literature, 
they do appear in the following: for the positions of the roots of a random polynomial (see Schehr and Majumdar~\cite{SM1, SM2}), and for a predator-prey model with resetting (see Evans, Majumdar, and Schehr~\cite{EMS22}).}

\subsection{Discussion of our assumptions}
As in~\cite{MUB18}, our proof relies on the fact that, by the Markov property, for all $t\geq 0$, $X(t) = Z(S(t))$ in distribution, where, on the right-hand side, the process $(S(t))_{t\geq 0}$ is independent of the Markov process~$Z$ of semi-group~$P$. The process $S(t)$ is a sum of two components, one of which is the time from $t$ back to the last relocation time before~$t$; we call this quantity $A(t)$. The fact that $A(t)/s(t)\to 0$ almost surely under {\bf (A1a)} and only in probability if {\bf (A1b)} explains why~\eqref{eq:Result} holds almost surely under {\bf (A1a)} and in probability under {\bf (A1b)}.

Assumption {\rm\bf (A2)} allows for a large class of semi-groups. 
For example, the semi-group of a random walk with i.i.d.\ increments satisfies {\rm\bf (A2)} 
as long as the Laplace transform of the increments has infinite radius of convergence. 
Similarly, the semi-group of the standard $d$-dimensional Brownian motion 
satisfies {\rm\bf (A2)} with $\Lambda_Z(\zeta) = \|\zeta\|^2/2$.

{\cec Assumption~{\rm\bf (A3)} is necessary to apply Gartner-Ellis theorem, 
which is how we prove Theorem~\ref{ThLarDevMW}.
Without Assumptions~{\rm\bf (A3)}(i-iii), 
one would still get some large deviation result 
(see the conclusions (i) and (ii) of Gartner-Ellis theorem - Theorem~\ref{ThGEas}). 
However, Assumption~{\rm\bf (A3)}(o) seems more crucial to getting a large deviation result at all.
Examples of run-length distributions $\phi$ and Markov processes $Z$ that satisfy Assumption~{\rm\bf (A3)} are:
\begin{itemize}
\item If $\int \mathrm e^{\xi x}\mathrm d\phi(x)<\infty$ for all $\xi\in\mathbb R$, 
$\Lambda(\xi)$ is (finite and) differentiable on $\mathbb R$,
and $\Lambda_Z$ is (finite and) differentiable on $\mathbb R^d$, 
then Assumption~{\rm\bf (A3)} holds. 
For example, one can take, for $Z$, the standard $d$-dimensional Brownian motion, 
and $\phi$ with bounded support or such that $\phi([x,\infty)) = \exp(-x^{\kappa})$ for some $\kappa>1$.
\item If $\phi$ is the exponential distribution with parameter $\lambda>1$, and $Z$ is the standard $d$-dimensional Brownian motion, then Assumption~{\rm\bf (A3)} holds. Indeed, in that case,
\[\Lambda(\xi) = \begin{cases}
\frac{\xi}{\lambda-\xi} &\text{ if }|\xi|<\lambda\\
\infty &\text{ otherwise.}
\end{cases}\]
Thus, $\mathcal D = \{\zeta\in\mathbb R^d\colon \|\zeta\|<\sqrt{2\log \lambda}\}$.
Assumption~{\rm\bf (A3)}(o) thus holds.
Importantly, note that, if $\lambda\leq 1$, then Assumption~{\rm\bf (A3)}(o) fails.
Furthermore, $\Lambda$ is lower semi-continuous on $\mathbb R$, and $\Lambda_Z$ is continuous on $\mathbb R^d$; thus, Assumption~{\rm\bf (A3)}(i). Because $\Lambda$ is differentiable on $(-\lambda, \lambda)$ and $\Lambda_Z$ is differentiable on $\mathbb R^d$, Assumption~{\rm\bf (A3)}(ii) holds.
Finally,
$\Lambda'(\xi) = \lambda/{(\lambda-\xi)^2}$ diverges when $|\xi|\uparrow\lambda$. Thus, Assumption~{\rm\bf (A3)}(iv) holds.
\item In the original model of Boyer and Solis-Salas~\cite{BSS14}, i.e.\ if $\phi$ is the geometric distribution with success parameter $q\in(0,1)$ and $Z$ is the simple symmetric $d$-dimensional random walk, then Assumption~{\rm\bf (A3)} holds as soon as $q>1-\nicefrac1{\mathrm{e}}$.
Indeed, in this case, one can check that
\[\Lambda(\xi) = \begin{cases}
\frac q{\xi} \left(\frac{q\mathrm e^{\xi}}{1-(1-q)\mathrm e^{\xi}}-1-\frac\xi q\right)&\text{ if }\xi<\log\Big(\frac1{1-q}\Big)\\
\infty &\text{ otherwise.}
\end{cases}\]
Furthermore, $\Lambda_Z(\zeta) = \mathbb E[\mathrm e^{\zeta \cdot X}]$, where $X$ is an increment of the simple symmetric random walk, and thus $\Lambda_Z(0) = 1$. 
Thus, Assumption {\bf (A3)}(o) only holds if $1<\log(1/(1-q))$, 
which is equivalent to $q>1-\nicefrac1{\mathrm{e}}$. 
It is straightforward to check that, in that case, Assumptions {\bf (A3)}(i-iii) also hold.
\end{itemize}
As one can see with these examples, Assumption~{\bf (A3)} ensures that the tail of the run-length distribution is not too heavy, relatively to the large deviations of the underlying process $Z$.}

{\cec \medskip
{\bf Remark on the rate function:} 
As often in large deviations, it is difficult to get an explicit formula for the rate function.
In fact, even in simple examples ($\phi = \delta_1$ or $\phi = \mathrm{Exp}(\lambda>1)$, and $Z$ being the standard one-dimensional Brownian motion), we were not able to calculate an explicit formula.}

\section{Preliminaries}

\subsection{The weighted random recursive tree}
As discussed in~\cite{MUB18}, there is a link between the monkey process and 
the weighted random recursive tree ({\sc wrrt}) with weights $(W_i)_{i\geq 1}$ given by
\begin{equation}\label{eq:defW1}
W_i = \int_{T_{i-1}}^{T_i} \mu(s)\,\mathrm ds\qquad (i\geq 1).
\end{equation}
This tree is defined recursively as follows: 
$\tau_0$ has only one node, the root $\nu_1$, which has weight $W_1$. 
Given $\tau_{n-1}$, we pick a node at random in $\tau_{n-1}$ 
with probability proportional to the weights, 
and add node $\nu_n$ with weight $W_n$ as a new child to this randomly chosen node; 
this defines~$\tau_n$.

The weighted random recursive tree for a general weight sequence $(W_i)_{i\geq 1}$
was originally introduced by Borovkov and Vatutin~\cite{BV05, BV06}.
Note that when all weights almost surely equal~1, 
the {\sc wrrt} is the classical random recursive tree.
More recently, Mailler and Uribe-Bravo~\cite{MUB18} 
proved the convergence in probability of its ``profile'' 
(a measure encoding how many nodes are in each generation of the tree);
there result applies in particular to the case when the weights are i.i.d.\ random variables, under an 8-th moment condition.
S\'enizergues~\cite{D19} 
proved results about its degree distribution and height for 
a deterministic sequence of weights under very general assumptions 
(with more refined results about the height --maximum distance of a node to the root-- in Pain and S\'enizergues~\cite{PS}).
Finally, Fountoulakis, Iyer, Mailler and Sulzbach~\cite[Prop.\ 3]{FIMS} (see also Iyer~\cite{Tejas} for a generalisation) and Lodewijks and Ortgiese~\cite{Bas} prove asymptotic results on, respectively the degree distribution and the largest degree of the {\sc wrrt} with i.i.d.\ weights.

\subsection{Preliminaries to the proof of Theorem~\ref{ThLarDevMW}}
As mentioned in the introduction, the key of the proof is to write that, for all $t\geq 0$, in distribution,
$X(t) = Z(S(t))$. The aim of this section is to explain why this identity is true, and to give a formula for~$S(t)$; all the statements of this section are already proved in~\cite{MUB18}; we refer the reader to Figure~\ref{fig:eqd} where the notation is illustrated.

Fix $t\geq 0$, and set $i(t)$ to be the number of the ``run'' $t$ belongs to, i.e.\ 
\begin{linenomath}\begin{equation}\label{eq:def_i(t)}
\sum_{i=1}^{i(t)-1} L_i\leq t<\sum_{i=1}^{i(t)} L_i.
\end{equation}\end{linenomath}
We define a sequence $(i_k(t))_{0\leq k\leq K(t)}$ recursively as follows:
set $i_0(t) = i(t)$ and if $i(t) = 1$ then set $K(t) = 0$. 
For all $k\geq 0$ given $i_k(t)$ and assuming that $k\neq K(t)$,
we set $i_{k+1}(t)$ as the number of the run that contains $R_{i_k(t)}$, i.e.
\[\sum_{i=1}^{i_{k+1}(t)-1} L_i\leq R_{i_k(t)}<\sum_{i=1}^{i_{k+1}(t)} L_i;\]
if $i_{k+1}(t)=1$ then set $K(t) = k$.

\begin{figure}
\begin{center}\includegraphics[width=10cm]{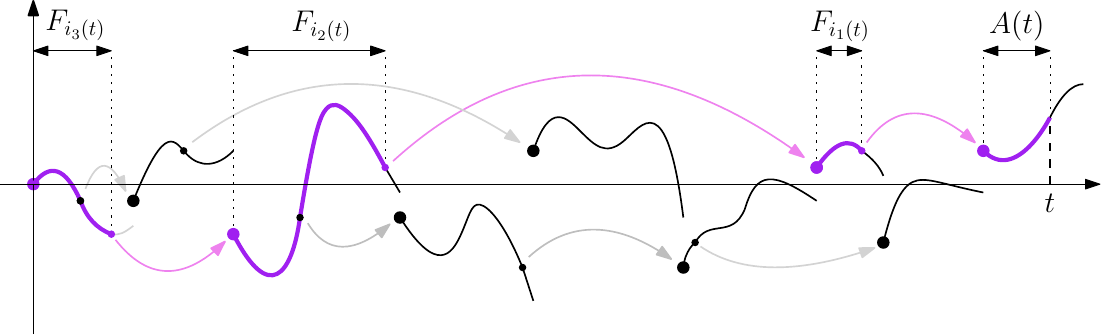}\end{center}
\caption{A trajectory of the monkey process up to time $t$ where we illustrate the definitions of $A(t), (i_k(t))_{0\leq k\leq K(t)}$, and $(F_i)_{i\geq 1}$. In this case, $K(t) = 3$, $i_0(t) = 9$ because time $t$ falls into the 9th run. We also have $i_1(t) = 7$, $i_2(t) = 3$ and $i_3(t) = 1$. 
Collapsing the pink arrows, one can glue the purple trajectory together and ignore the rest of the monkey process trajectory. 
By the Markov property of the underlying Markov process, 
this purple trajectory is distributed as $Z$ run for a random amount of time, 
which equals $S(t) = A(t) + \sum_{k=1}^{K(t)} F_{i_k(t)}$.}
\label{fig:eqd}
\end{figure}

Using the fact that, by definition of the model, $X(R_{i_k(t)}) = X(T_{i_k(t)-1})$, and by the Markov property of $Z$ we get that $X(t) = Z(S(t))$ in distribution (see Figure~\ref{fig:eqd}), where $S$ and $Z$ are independent and
\begin{linenomath}\begin{equation}\label{eq:defS}
S(t) = A(t) + \sum_{k=1}^{K(t)} F_{i_k(t)},
\end{equation}\end{linenomath}
where $A(t) = t - \sum_{i=1}^{i(t)-1} L_i$ is the time from $t$ back to the last relocation time from~$t$
and, for all $i\geq 1$, for all $0\leq x \leq L_i$,
\begin{linenomath}\begin{equation}\label{eq:defF}
\mathbb P(F_i\leq x \mid \bs L) 
= \frac{\int_{T_{i-1}}^{T_{i-1}+x} \mu(u)\,\mathrm du}{\int_{T_{i-1}}^{T_{i-1}+L_i} \mu(u)\,\mathrm du}.
\end{equation}\end{linenomath}
This is true because $F_i$ is distributed as the distance from $R_n$ to $T_{i-1}$ given that $R_n\in[T_{i-1}, T_i)$ (note that this distribution, and thus $F_i$, does not depend on~$n$).

\begin{remark} In the uniform memory case, the distribution of $F_i$ is easy to describe: in distribution, $F_i = U_i L_i$, where $(U_i)_{i\geq 1}$ is a sequence of i.i.d.\ random variables, uniformly distributed on~$[0,1]$.
 \end{remark}

It is important to note that the distribution of the $F_{i_k(t)}$'s is not the same as the distribution of the $F_i$'s, because the index $i_k(t)$ is random and biases towards longer runs: we expect the $F_{i_k(t)}$'s to be stochastically larger than the $F_i$'s. 
To explicit this bias, we need to introduce the following notion of genealogy between runs: 
For two integers $i, j\geq 1$, we say that run $i$ is the parent of run $j$ 
if $R_j$ belongs to the $i$-th run, i.e.\ $R_j\in [T_{i-1}, T_i)$;
we write $i\prec j$ if run $i$ is an ancestor of run $j$ for this notion of parenthood.
E.g., on Figure~\ref{fig:def}, we have $1\prec 2\prec 5$.
Note that the genealogical tree for this notion of parenthood is 
the {\sc wrrt} of weight sequence $\bs W = (W_i)_{i\geq 1}$ 
(see~\eqref{eq:defW1} for the definition of this sequence).

With this notation, Equation~\eqref{eq:defS} becomes
\begin{linenomath}\begin{equation}\label{eq:formS}
S(t) = A(t) + \sum_{i=1}^{i(t)-1} F_i\mathds 1_{i\prec i(t)}.
\end{equation}\end{linenomath}
In~\cite{MUB18}, the authors prove the following result, which we use extensively in our proofs:
\begin{proposition}[See \cite{MUB18}]\label{prop:dobr}
Conditionally on $\bs L$, for all $n\geq 1$, 
the random variables $(\mathds{1}_{i\prec n})_{1\leq i\leq n-1}$ 
are independent random variables of respective parameters $W_i/S_i$, 
where, for all $i\geq 1$,
\begin{linenomath}\begin{equation}\label{eq:defW}
W_i = \int_{T_{i-1}}^{T_i}\mu(u)\,\mathrm du 
\quad \text{ and }\quad 
S_i=\sum_{k=1}^i W_k.
\end{equation}\end{linenomath}
\end{proposition}

\begin{remark}
The proof of this proposition relies on the fact that the genealogical tree defined by the parenthood relationship $\prec$ is the weighted random recursive tree of weight sequence $(W_n)_{n\geq 1}$. The proof of Proposition~\ref{prop:dobr} (see~\cite{MUB18}) is a generalisation of the proof of the same result for the random recursive tree, which dates back to Dobrow~\cite{Dobrow}.
\end{remark}

\section{Proof of Theorem~\ref{ThLarDevMW}}
\subsection{G\"artner-Ellis' theorem}
The proof of Theorems~\ref{ThLarDevMW} relies on G\"artner-Ellis' large deviations theorem (see~\cite[Chapter~V2]{H08}):
\begin{dfn}\label{def:Legendre}
We let $F^*$ denote the Legendre transform of any function~$F : \mathbb R^d \to \RR$, 
i.e. 
\[F^*(x)=\sup_{t\in \RR^d} \{x\cdot t -F(t)\}, \quad \forall x\in \RR^d.\]
A value $t\in \RR^d$ is called an exposing hyperplane for $x\in\mathbb R^d$ if 
\[F^*(y)-F^*(x)>(y-x)\cdot t, \quad \forall y\neq x.\]
A point $x\in \RR^d$ is called exposed if there exists at least one exposing hyperplane for it.
\end{dfn}

Let $(X(t))_{t\geq 0}$ be a real-valued stochastic process on some probability space $(\Omega, \mathcal{F},\PP)$. 
We assume that there exists a sequence $(u_t)_{t\geq 0}$ such that $u_t\to +\infty$ when $t\to \infty$ and 
\begin{itemize}
\item[{\bf (H1)}] there exists a function $F: \RR \to [-\infty,+\infty]$ such that, 
for all $\zeta\in\mathbb R^d$, 
\[\lim_{t\to \infty} \frac{1}{u_t}\log \mathbb E[\mathrm e^{u_t\zeta\cdot X(t)}]=F(\zeta).\]
\item[{\bf (H2)}] $0\in \inte(\mathcal D_{F})$, where $\mathcal D_{F}=\{ t\in \RR : F(t)<\infty\}.$
\end{itemize}

\begin{theorem}\label{ThGEas} 
Let $(X(t))_{t\geq 0}$ be a stochastic process taking values in $\mathbb R^d$ and satisfying Assumptions~{\rm\bf (H1-2)}. 
Then 
\begin{itemize}
\item[{\rm (i)}] $\limsup_{t\to \infty} \frac{1}{u_t}\log \mathbb P(X(t)\in \mathcal C)\leq -\inf_{\mathcal C}F^*$ for all closed set $\mathcal C\subseteq \RR^d$.
\item[{\rm (ii)}] $\liminf_{t\to \infty} \frac{1}{u_t}\log \mathbb P(X(t)\in \mathcal G)\geq-\inf_{\mathcal G\cap \mathcal E} F^*$ for all open set $\mathcal G\subseteq \RR^d$, 
where $\mathcal E=\mathcal E(F,F^*)$ 
is the set of exposed points of $F^*$ 
whose exposing hyperplanes belongs to $\inte(D_{F})$.
\item[{\rm (iii)}] Suppose, in addition, that $F$ satisfies the following: \begin{itemize}
\item[a.] $F$ is lower semi-continuous on $\RR^d$.
\item[b.] $F$ is differentiable on $\inte(\mathcal{D}_{F})$.
\item[c.] Either $\mathcal D_{F}=\RR^d$ or $\lim_{t\to \partial \mathcal D_{F}} |F'(t)|=\infty$.
\end{itemize}
Then, $\mathcal G\cap \mathcal E$ can be replaced with $\mathcal G$ in the right-hand side of {\rm (ii)}.
\end{itemize}
\end{theorem}

\begin{remark}\label{rk:GE}
This version is in fact an extension of G\"artner-Ellis' theorem as stated in, 
e.g.~\cite[Chapter~V2]{H08}: in the classical version, $u_t = t$ for all $t>0$.
A simple change of the time-variable~$t$ gives the extended version stated here.

Also note that if the distribution of the process $(X(t))_{t\geq 0}$ is random,
and if {\rm\bf (H1-2)} hold almost surely, then the conclusion of Theorem~\ref{ThGEas} holds almost surely as well. It is less immediate but easy to check by adapting the proof of, e.g., \cite[Chapter~V2]{H08} that, if {\rm\bf (H1-2)} hold in probability, then the conclusion of Theorem~\ref{ThGEas} also holds in probability.
\end{remark}

\subsection{Proof of Theorem~\ref{ThLarDevMW}}
To prove Theorem~\ref{ThLarDevMW}, 
we first prove that $(S(t))_{t\geq 0}$ satisfies {\rm\bf (H1-2)} with $u_t = s(t)$ and $F=\Lambda$ (see Equation~\eqref{eq:def_Lambda} for the definition of $\Lambda$, and Equation~\eqref{eq:st} for the definition of $s(t)$). 
Recall that, by Equation~\eqref{eq:formS}, for all $t\geq 0$,
$S(t) = A(t) + B(t)$, where 
\begin{linenomath}\begin{equation}\label{eq:defB}
B(t) = \sum_{i=1}^{i(t)-1} F_i\mathds 1_{i\prec i(t)}.
\end{equation}\end{linenomath}
The following proposition implies that, conditionally on $\bs L$, 
almost surely, $(B(t))_{t\geq 0}$ satisfies Assumption {\rm\bf (H1)}:
\begin{proposition}\label{phi for S(t)} 
Suppose that $\int_0^{\infty} \mathrm e^{\xi u}\,\mathrm d\phi(u)<\infty$ for all $\xi \in \RR$. 
Then, almost surely,
\[\lim_{t\to \infty} \frac{1}{s(t)}\log \EE[\mathrm{e}^{\xi B(t)}\mid \bs L] = \Lambda(t),
\]
where $\Lambda$ is defined in Equation~\eqref{eq:def_Lambda} 
and $s(t)$ in Equation~\eqref{eq:st}.
\end{proposition}
Because it is quite technical, we postpone the proof of Proposition~\ref{phi for S(t)} to Section~\ref{sec:proof_props}. We now treat the term $A(t)$ in $S(t) = A(t)+B(t)$; 
recall that $A(t)$ is the time back from $t$ to the last relocation time before~$t$: $A(t) = t-T_{i(t)}$.

\begin{lemma}\label{A(t) as} 
Suppose that $\int_0^{\infty} \mathrm e^{\xi u}\,\mathrm d\phi(u)<\infty$ for all $\xi \in \RR$.
Under Assumption {\rm\bf (A1)}, $A(t) = o(s(t))$ as $t\to \infty$, almost surely under {\rm\bf (A1a)} and in probability under {\rm\bf (A1b)}.
\end{lemma}

\begin{proof} 
First recall that $A(t)$ converges in distribution to an almost-surely 
finite random variable (see, e.g.~\cite{Bertoin}). 
This implies that $A(t)/s(t)\to 0$ in probability for all memory kernels satisfying Assumption~{\bf (A1)}.

Under Assumption {\rm\bf (A1b)} on the memory kernel, 
and by definition of $s(t)$ (see Equation~\eqref{eq:st})
we have $\log t=\mathcal O(s(t))$ as $t\to+\infty$. 
Therefore, it suffices to show that $A(t)=o(\log t)$ almost surely when $t\to+\infty$.

First note that, by definition, $A(t)\leq L_{i(t)}$, where we recall that $i(t)$ is the number of the run to which $t$ belongs (see Equation~\eqref{eq:def_i(t)}). Equation~\eqref{eq:def_i(t)} also implies $t\geq \sum_{i=1}^{i(t)-1} L_i$. These two inequalities imply
\[\frac{A(t)}{\log t}\leq \frac{L_{i(t)}}{\log \big(\sum_{i=1}^{i(t)-1} L_i\big)}.\] 
By the law of large numbers
\[\log\left(\frac{ \sum_{i=1}^n L_{i}}{n}\right)\to \log \EE L,\] almost surely when~$n$ tends to infinity,
where $L$ is a random variable of distribution~$\phi$, the joint distribution of the run-lengths.
Equivalently, we have $\log \left(\sum_{i=1}^n L_i\right) \sim \log(n\EE L)$ 
almost surely as~$n$ tends to infinity. 
Since $i(t)$ tends to infinity almost surely, 
we only need to prove that $L_{n+1}/\log(n\EE L)\to 0$ almost surely when $n$ tends to infinity.

Fix $\varepsilon>0$. Using Markov's inequality, we get that, for any $\xi >0$,
\[
\PP\left( L_{n+1} \geq \varepsilon \log(n\EE L\right))
= \PP\big(\mathrm{e}^{\xi L_{n+1}} \geq (n\EE L)^{\varepsilon \xi}\big)
\leq \frac{\EE[\mathrm{e}^{\xi L}]}{n^{\varepsilon \xi}\EE[L]^{\varepsilon \xi}}.
\]
Choosing $\xi = 2/\varepsilon$, we have that \[\sum_{i=1}^{\infty} \PP\left( \frac{L_{n+1}}{\log (n\EE L)}\geq \varepsilon \right) \leq \sum_{i=1}^{\infty} \frac{\EE[\mathrm{e}^{2 L/\varepsilon}]}{n^2(\mathbb EL)^2}<\infty,\]
because, by assumption, $\mathbb EL>0$ and $\mathbb E\mathrm e^{\xi L}<+\infty$ for all $\xi\in\mathbb R$.
By Borel-Cantelli's lemma we get that $L_{n+1}/\log(n\EE L)\to 0$ almost surely when $n$ tends to infinity, which concludes the proof.
\end{proof}

\begin{remark}
If we assume that $\EE[\exp(\mathrm{e}^{\xi L})]<\infty$ for all $\xi \in \RR$, 
then $A(t) = o(s(t))$ almost surely also under {\rm\bf (A1b)}, and thus Equation~\eqref{eq:Result} of Theorem~\ref{ThLarDevMW} holds almost surely under {\rm\bf(A1b)}.
\end{remark}

From Proposition~\ref{phi for S(t)} and Lemma~\ref{A(t) as}, we get:
\begin{proposition} \label{prop:LDP_S}
Let $S(t)$ and $s(t)$ as in \eqref{eq:formS} and \eqref{eq:st} respectively. 
Then, for all $\xi\in\mathbb R$,
\[\lim_{t\to+\infty} \frac1{s(t)} \log \EE[\mathrm{e}^{\xi S(t)}\mid \bs L]=\Lambda(\xi),
\]
almost surely under assumption {\rm\bf (A1a)} and in probability under {\rm\bf (A1b)}.
\end{proposition}

\begin{proof} 
For all $\xi\in\RR$, using the fact that $S(t) = A(t) + B(t)$ where $B(t)$ is defined in Equation~\eqref{eq:defB}, we get
\[
\log \EE[\mathrm{e}^{\xi S(t)}\mid \bs L] = \log \EE[\mathrm{e}^{\xi(A(t)+B(t))}\mid \bs L].\]
Note that $A(t)$ is, by definition, $\bs L$-measurable, which implies
\[\frac1{s(t)}\log \EE[\mathrm{e}^{\xi S(t)}\mid \bs L] 
=\frac1{s(t)} \log \EE[\mathrm{e}^{\xi B(t)} \mid\bs L]
+\frac{A(t)}{s(t)}
\to \Lambda(\xi),
\]
as $t\to \infty$ by Proposition~\ref{phi for S(t)} and Lemma~\ref{A(t) as}. We also know by Lemma~\ref{A(t) as}, that the above convergence holds almost surely under assumption {\rm\bf (A1a)} and in probability under {\rm\bf (A1b)}.
This concludes the proof.
%
\end{proof}

We are now ready to prove the quenched large deviation principle for the monkey process:
\begin{proof}[Proof of Theorem~\ref{ThLarDevMW}] 
We aim at applying Theorem~\ref{ThGEas} 
to the sequence $(X(t))_{t\in\mathcal T}$ conditionally on~$\bs L$. 
We first prove that, for all $\zeta \in \RR^d$, almost surely,
\[\lim_{t\to \infty} \frac1{s(t)}\log \EE[ \mathrm{e}^{\zeta\cdot X(t)}\mid \bs L]
=\Lambda_X(\zeta).\]
Let $\varepsilon>0$. 
By Assumption {\bf (A2)}, 
there exists $t_\varepsilon\in\mathcal T$ such that for all $t\geq t_\varepsilon$, 
\begin{equation}\label{eq:up_bound_Z}
\EE[\mathrm{e}^{\zeta\cdot Z(t)}] \leq \mathrm{e}^{(\Lambda_Z(\zeta)+\varepsilon)t}.
\end{equation}
Since $X(t) =  Z(S(t))$ in distribution, with $Z$ and $S$ independent on the right-hand side, we get
\begin{equation}\label{eq:sum_t0}
\EE[\mathrm{e}^{\zeta\cdot X(t)}\mid \bs L]
= \EE[\mathrm{e}^{\zeta\cdot Z(S(t))}\mid \bs L]
=\EE\big[\EE [\mathrm{e}^{\zeta\cdot Z(S(t))} \mid S(t)]\mathds 1_{S(t)\geq t_\varepsilon}\mid\bs L\big]
+\EE\big[\mathrm{e}^{\zeta\cdot Z(S(t))}\mathds 1_{S(t)<t_\varepsilon}\mid\bs L\big].
\end{equation} 
By~\eqref{eq:up_bound_Z}, almost surely for all $t\geq 0$,
\[\EE [\mathrm{e}^{\zeta\cdot Z(S(t))} \mid S(t)]\mathds 1_{S(t)\geq t_\varepsilon}
\leq \mathrm{e}^{(\Lambda_Z(\zeta)+\varepsilon)S(t)}.\]
Taking the conditional expectation with respect to $\bs L$, we get
\[\EE[\mathrm{e}^{\zeta\cdot Z(S(t))}\mathds 1_{S(t)\geq t_\varepsilon}\mid \bs L] 
\leq \EE[ \mathrm{e}^{(\Lambda_Z(\zeta)+\varepsilon)S(t)} \mid \bs L ].\]
Using Proposition~\ref{prop:LDP_S}, 
we get that almost surely for all~$t$ large enough,
\begin{equation}\label{eq:largerthan_t0}
\frac1{s(t)}\log \EE[\mathrm e^{\zeta\cdot Z(S(t))}\mathds 1_{S(t)\geq t_\varepsilon}\mid \bs L] 
\leq \frac1{s(t)}\log \EE[\mathrm e^{(\Lambda_Z(\zeta)+\varepsilon)S(t)}\mid \bs L]
\leq \Lambda(\Lambda_Z(\zeta)+\varepsilon)+\varepsilon.
\end{equation}
Moreover, almost surely when $t\to+\infty$,
\[\mathbb E[\mathrm e^{\zeta\cdot Z(S(t))}\mathds 1_{S(t) < t_\varepsilon}\mid\bs L]
\leq \mathbb P(S(t) < t_\varepsilon\mid\bs L)
\max_{u\in [0, t_\varepsilon)} \mathbb E[\mathrm e^{\zeta\cdot Z(u)}]\to 0\]
The last limit holds because, conditionally on $\bs L$, $S(t) \to+\infty$ almost surely, 
and because $\max_{u\in [0, t_\varepsilon)} \mathbb E[\mathrm e^{\zeta\cdot Z(u)}]<+\infty$ by Assumption {\bf (A2)} (Equation~\eqref{eq:upper_bound_Z}). Together with Equations~\eqref{eq:sum_t0} and~\eqref{eq:largerthan_t0}, this gives that, 
almost surely for all $t$ large enough,
\ba\frac1{s(t)}\log \EE[\mathrm e^{\zeta\cdot X(t)}\mid \bs L] 
&\leq \frac1{s(t)}\log \big(\mathrm e^{s(t)(\Lambda(\Lambda_Z(\zeta)+\varepsilon)+\varepsilon)}+\varepsilon\big)\\
&= \Lambda(\Lambda_Z(\zeta)+\varepsilon)+\varepsilon 
+ \frac1{s(t)}\log\big(1+\varepsilon\mathrm e^{-s(t)(\Lambda(\Lambda_Z(\zeta)+\varepsilon)+\varepsilon)}\big).
\ea
Because, by definition, $\Lambda\geq 0$ (see Equation~\eqref{eq:def_Lambda}),
we get that, for all $\varepsilon>0$, for all $\zeta\in\mathbb R^d$, 
almost surely for all $t$ large enough,
\[\frac1{s(t)}\log \EE[\mathrm e^{\zeta\cdot X(t)}\mid \bs L] 
\leq \Lambda(\Lambda_Z(\zeta)+\varepsilon)+2\varepsilon.\]
Similarly, one can show that, for all $\varepsilon>0$, 
for all $\zeta\in\mathbb R^d$, 
almost surely for all $t$ large enough, 
\[\frac{1}{s(t)}\log \EE[\mathrm e^{\zeta\cdot X(t)}\mid \bs L] 
\geq \Lambda(\Lambda_Z(\zeta)-\varepsilon)-2\varepsilon.\]
Because, by Assumption {\bf (A3)}, $\Lambda$ is continuous, we get that, 
almost surely for all $\zeta\in\mathbb R^d$,
\begin{equation}\label{eq:cv_lambda_X}
\frac{1}{s(t)}\log \EE[\mathrm e^{\zeta\cdot X(t)}\mid \bs L] \to \Lambda(\Lambda_Z(\zeta)),
\end{equation}
and the monkey process thus satisfies Assumption~{\bf (H1)} of Theorem \ref{ThGEas}. 
{\cec Assumption {\bf (H2)} holds because of Assumption {\bf (A3)}(o).
To conclude the proof, it only remains to check that $\Lambda_X = \Lambda\circ \Lambda_Z$ satisfies Assumptions (a-c) of Theorem \ref{ThGEas}(iii). 
This is true because of Assumptions {\bf (A3)}(i-iii).}
\end{proof}

\section{Proof of Proposition~\ref{phi for S(t)}} 
\label{sec:proof_props}

The proof of Proposition~\ref{phi for S(t)} is divided in several technical lemmas.
The idea is the following: recall that
\[B(t) = \sum_{i=1}^{i(t)-1} F_i \mathds 1_{i\prec i(t)},\]
where $i(t)$ is the number of the run that contains~$t$ (see~\eqref{eq:def_i(t)}). 
Recall that the $F_i$'s are independent conditionally on $\bs L$ and their respective distributions are given by~\eqref{eq:defF}.
Also, by Proposition~\ref{prop:dobr}, conditionally on $\bs L$, the random variables $(\mathds 1_{i\prec i(t)})_{1\leq i<i(t)}$ are independent Bernoulli random variables of respective parameters $W_i/S_i$.
This means that
\[\mathbb E_{\bs L}[\mathrm e^{\xi B(t)}]
=\prod_{i=1}^{i(t)-1} \mathbb E_{\bs L}[\mathrm e^{\xi F_i \mathds 1_{i\prec i(t)}}]
=\prod_{i=1}^{i(t)-1} \bigg(1+\frac{W_i}{S_i}(\mathbb E_{\bs L}[\mathrm e^{\xi F_i}]-1)\bigg),
\]
where we use $\mathbb E_{\bs L}$ (and $\mathbb P_{\bs L}$) to denote $\mathbb E[\,\cdot\mid\bs L]$ (and $\mathbb P(\,\cdot\mid \bs L)$). Heuristically (this will be done rigorously in this section), since $W_i/S_i\to 0$ as $i\to+\infty$, we get that, for large $t$,
\[\log \mathbb E_{\bs L}[\mathrm e^{\xi B(t)}]
\approx \sum_{i=1}^{i(t)-1} \frac{W_i}{S_i}(\mathbb E_{\bs L}[\mathrm e^{\xi F_i}]-1)
\]
In Lemma~\ref{Le: AsymptoticExpe}, we give estimates for $\mathbb E_{\bs L}[\mathrm e^{\xi F_i}]$ for large index $i$. These estimates imply that
\[\log \mathbb E_{\bs L}[\mathrm e^{\xi B(t)}]
\approx \sum_{i=1}^{i(t)-1} \frac{\mu(T_{i-1})}{S_i}\Big(\frac{\mathrm e^{\xi L_i}-1-\xi L_i}\xi\Big).
\]
In Lemma~\ref{Lemma: AsympSeries}, we prove that one can replace the $T_i$'s in this sum by $i\mathbb EL$, thus $S_i$ becomes $\int_0^{i\mathbb EL} \mu$. 
Once this substitution is done, the sum in the last display becomes a sum of independent random variables to which strong laws of large numbers apply: we thus get
\[\log \mathbb E_{\bs L}[\mathrm e^{\xi B(t)}]
\approx \mathbb E\Big[\frac{\mathrm e^{\xi L}-1-\xi L}\xi\Big]
\sum_{i=1}^{i(t)-1} \frac{\mu((i-1)\mathbb EL)}{\int_0^{i\mathbb EL} \mu}.
\]
The last part of the proof is done in Lemma~\ref{Le: AsymptMoment}, where we show that
\[\sum_{i=1}^{i(t)-1} \frac{\mu((i-1)\mathbb EL)}{\int_0^{i\mathbb EL} \mu}
\approx 
\begin{cases}
\displaystyle\frac{s(t)}{\mathbb EL} & \text{ if }\mu=\mu_1\\[5pt]
\displaystyle\frac{s(t)}{\mathbb E[L]^{1-\delta}} & \text{ if }\mu=\mu_2,
\end{cases}\]
and this concludes the proof of Proposition~\ref{phi for S(t)}.

This section is devoted to making these heuristics rigorous:
In Section~\ref{sub:tech_lem}, 
we prove the three technical lemmas mentioned in the sketch of proof above, 
and in Section~\ref{sub:proof_prop33}, 
we show how they can be used to prove Proposition~\ref{phi for S(t)}.

\subsection{Technical lemmas}\label{sub:tech_lem}

\begin{lemma}\label{Le: AsymptoticExpe} Let $f: \RR \to \RR$ be a differentiable function. Then, \begin{linenomath}\begin{equation} \label{eq: Ri} W_i\EE_{\bs{L}}[f'(F_i)]=\mu(T_{i-1})R_i,\end{equation}\end{linenomath} 
and there exists a random integer $I_0$ and a constant $c>0$ (which is independent of~$f$)
such that, for all $i\geq I_0,$ 
\begin{equation} \label{ineq: Ri} 
\left| R_i-(f(L_i)-f(0))\right|\leq 
\begin{cases} 
\displaystyle\frac{c(f(L_i)-f(0))L_i(\log T_{i-1})^{\tilde{\alpha}-1}}{T_{i-1}} & \text{if} \; \mu=\mu_1 \\[12pt]
\displaystyle\frac{c(f(L_i)-f(0))L_i}{T_{i-1}^{1-\delta}} & \text{if} \; \mu=\mu_2,
\end{cases} \end{equation}
where $\tilde{\alpha}=\max\{1,\alpha\}$.
\end{lemma}

\begin{proof}First note that, by defintion of $(F_i)_{i\geq 1}$ (see~\eqref{eq:defF}),
for all $i\geq 1$,
\[\EE_{\bs{L}}[f'(F_i)]=\int_0^{L_i} f'(u) \frac{\mu(T_{i-1}+u)}{W_i}\; \mathrm{d}u.\] Therefore, $W_i\EE_{\bs{L}}[f'(F_i)]=\mu(T_{i-1})R_i,$ where $R_i = \int_0^{L_i} f'(u)\frac{\mu(T_{i-1}+u)}{\mu(T_{i-1})}\; \mathrm{d}u.$

Using Assumption~{\rm\bf (A1)}, Markov's inequality, and Borel-Cantelli's lemma, 
one can prove that $L_i/i^{1-\delta}\to 0$ almost surely when $i\to \infty$,
for all $\delta \in (0,\nicefrac12]$.
Moreover, the law of large numbers implies that $T_i\sim i\EE[L]$, 
and thus $L_i/T_i^{1-\delta}\to 0$, almost surely when $i\to+\infty$. 
This also implies that $L_i/T_i$ converges to $0$ almost surely when $i\to+\infty$.
We now treat the cases $\mu = \mu_1$ and $\mu=\mu_2$ separately:

(1) If $\mu=\mu_1$, then \begin{linenomath}\begin{equation} \label{Ri m1} R_i=\int_0^{L_i} \frac{f'(u)}{1+\nicefrac{u}{T_{i-1}}} \left( 1+\frac{\log(1+\nicefrac{u}{T_{i-1}})}{\log T_{i-1}}\right)^{\alpha -1} \exp\left\lbrace \beta (\log T_{i-1})^{\alpha} \left[ \left( 1+\frac{\log(1+\nicefrac{u}{T_{i-1}})}{\log T_{i-1}} \right)^{\alpha} -1 \right] \right\rbrace \; \mathrm{d}u.\end{equation}\end{linenomath} 
We first assume that $\alpha >1$: 
Since the last two terms in the product in the integral of Equation~\eqref{Ri m1} are at least equal to~$1$, 
we have 
\[R_i \geq \int_0^{L_i} \frac{f'(u)}{1+\nicefrac{L_i}{T_{i-1}}}\; \mathrm{d}u 
=\frac{f(L_i)-f(0)}{1+\nicefrac{L_i}{T_{i-1}}}.\]
Since $L_i/T_{i-1}\to 0$ almost surely as $i\to+\infty$,
there exists a random integer $I_1$ such that, for all $i\geq I_1$, $0\leq L_i/T_{i-1}\leq 1$, hence
 \[R_i\geq (f(L_i)-f(0))(1-c_1L_i/T_{i-1}),\]
 where $c_1=\sup_{x\in [0,1]}\left\lbrace \frac{1}{x}\left(1-\frac{1}{1+x}\right) \right\rbrace.$

Similarly, since $u\leq L_i$, $1+u/T_{i-1}\geq 1$, 
and the last two factors in the integral of Equation~\eqref{Ri m1} are increasing in $u$, 
we have 
\[R_i\leq (f(L_i)-f(0))\left( 1+\frac{\log(1+\nicefrac{L_i}{T_{i-1}})}{\log T_{i-1}}\right)^{\alpha -1} \exp\left\lbrace \beta (\log T_{i-1})^{\alpha} \left[ \left( 1+\frac{\log(1+\nicefrac{L_i}{T_{i-1}})}{\log T_{i-1}} \right)^{\alpha} -1 \right] \right\rbrace.\]
Almost surely, there exists $I_2>I_1$ such that, for all $i\geq I_2$, $\log T_{i-1}\geq 1$. Therefore for all $i\geq I_2$, \[0\leq L_i/T_{i-1}\leq 1\quad \text{and} \quad \frac{L_i}{T_{i-1}\log T_{i-1}}\leq 1.\] Hence for 
\[c_2=\sup_{x\in [0,1]}\left\lbrace \frac{1}{x}\log(1+x)\right\rbrace, 
\,c_3=\sup_{x\in [0,1]}\left\lbrace \frac{1}{x}((1+c_2x)^{\alpha} -1)\right\rbrace,
\,\text{ and }
c_4=\sup_{x\in [0,1]}\left\lbrace \frac{\beta}{x}((1+c_2x)^{\alpha} -1)\right\rbrace,\] we have \begin{linenomath}\begin{align*}
R_i & \leq (f(L_i)-f(0))\left( 1+\frac{c_2L_i}{T_{i-1}\log T_{i-1}}\right)^{\alpha -1} \exp\left\lbrace \beta (\log T_{i-1})^{\alpha} \left[ \left( 1+\frac{c_2L_i}{T_{i-1}\log T_{i-1}} \right)^{\alpha} -1 \right] \right\rbrace \\
& \leq (f(L_i)-f(0))\left( 1+\frac{c_3L_i}{T_{i-1}\log T_{i-1}}\right) \exp\left\lbrace \frac{c_4L_i(\log T_{i-1})^{\alpha -1}}{T_{i-1}}   \right\rbrace.
\end{align*}\end{linenomath}
Almost surely there exists $I_3>I_2$ such that, for all $i\geq I_3$, $\frac{L_i(\log T_{i-1})^{\alpha -1}}{T_{i-1}}\leq 1$. Therefore, for all $i\geq I_3$, 
\[R_i \leq (f(L_i)-f(0))\left( 1+\frac{c_3L_i}{T_{i-1}\log T_{i-1}}\right)\left( 1+ \frac{c_5L_i(\log T_{i-1})^{\alpha -1}}{T_{i-1}} \right),\] where $c_5=\sup_{x\in [0,1]} \left\lbrace \frac{1}{x}(\mathrm e^{c_4x}-1)\right\rbrace$. 
Finally, for $c=c_3+c_5+c_3c_5$, we get 
\[R_i\leq (f(L_i)-f(0))\left( 1+ \frac{cL_i(\log T_{i-1})^{\alpha -1}}{T_{i-1}} \right),\]
which concludes the case when $\mu=\mu_1$ and $\alpha>1$.

We can prove the result in the same way when $\alpha\leq 1$.
Instead of calculating constants explicitly, we use the $\mathcal O$-notation in the following sense: we say that $v_i = \mathcal O(u_i)$ if there exists an $\bs L$-measurable random integer~$I$ and a deterministic constant $C$ such that 
\begin{equation}\label{eq:O}
|v_i|\leq Cu_i\text{ for all }i\geq I.
\end{equation}
With this notation, when $\alpha\leq 1$, for all $i\geq I_2$,
\begin{linenomath}\begin{align*}
R_i & \geq (f(L_i)-f(0))\cdot \frac{1}{1+\nicefrac{L_i}{T_{i-1}}}\left( 1+\frac{\log(1+\nicefrac{L_i}{T_{i-1}})}{T_{i-1}}\right)^{\alpha-1} \\
& = (f(L_i)-f(0))\bigg(1+\mathcal{O}\bigg(\frac{L_i}{T_{i-1}}\bigg)\bigg)\bigg(1+\mathcal{O}\bigg(\frac{L_i}{T_{i-1}\log T_{i-1}}\bigg)\bigg) \\
& = (f(L_i)-f(0))\bigg(1+\mathcal{O}\bigg(\frac{L_i}{T_{i-1}}\bigg)\bigg).
\end{align*}\end{linenomath}
Similarly,
\begin{linenomath}\begin{align*}
R_i & \leq (f(L_i)-f(0)) \exp\left\lbrace \beta (\log T_{i-1})^\alpha\left[ \left( 1+\frac{\log(1+\nicefrac{L_i}{T_{i-1}})}{\log T_{i-1}} \right)^{\alpha -1} \right] \right\rbrace \\
& \leq (f(L_i)-f(0)) \bigg(1+\mathcal{O}\bigg(\frac{L_i}{T_{i-1}(\log T_{i-1})^{\alpha -1}}\bigg)\bigg),
\end{align*}\end{linenomath}
which concludes the case when $\mu=\mu_1$ and $\alpha\leq 1$.

(2) If $\mu=\mu_2$, then \[R_i=\int_0^{L_i} f'(u) \left(1+\frac{u}{T_{i-1}}\right)^{\delta -1} \exp\left\lbrace \gamma T_{i-1}^{\delta} \left[ \left(1+\frac{u}{T_{i-1}}\right)^{\delta} -1 \right] \right\rbrace \mathrm{d}u.\]
Since $\delta<1$, $\gamma\geq 0$, and $0\leq u\leq L_i$, we have 
\[R_i\geq (f(L_i)-f(0))\left(1+\frac{L_i}{T_{i-1}}\right)^{\delta-1}.\]
Recall that, for all $i\geq I_1$, $0\leq \nicefrac{L_i}{T_{i-1}}\leq 1$. 
Hence, if $c_6=\sup_{x\in [0,1]}\left\lbrace \frac{1}{x}(1-(1+x)^{\delta-1} \right\rbrace$, then
\[R_i\geq (f(L_i)-f(0))\left( 1-c_6\frac{L_i}{T_i}\right).\]
For the upper bound, \[R_i\leq (f(L_i)-f(0))\exp\left\lbrace \gamma T_{i-1}^{\delta} \left[ \left(1+\frac{L_i}{T_{i-1}}\right)^{\delta} -1 \right] \right\rbrace.\]
For all $i\geq I_1$, $R_i\leq (f(L_i)-f(0)) \exp(\gamma c_7T_{i-1}^{\delta-1}L_i)$, where $c_7=\sup_{x\in [0,1]}\left\lbrace \frac{1}{x}((1+x)^{\delta}-1) \right\rbrace$. Since $T_{i-1}^{\delta-1}L_i\to 0$ a.s. when $i\to \infty$, we can find a random integer $I_3>I_1$ such that, for all $i\geq I_3$, $\gamma c_7 T_{i-1}^{\delta-1}L_i\leq 1$, thus \[R_i\leq (f(L_i)-f(0))\left(1+\frac{c_8L_i}{T_{i-1}^{1-\delta}}\right),\]
where $c_8=\sup_{x\in [0,1]} \left\lbrace \frac{1}{x}(\mathrm e^{\gamma c_7x}-1) \right\rbrace$.
\end{proof}

\begin{lemma}\label{Lemma: AsympSeries} Let $g:\RR \to \RR$ be a function such that $\EE[g(L)],\Var(g(L))<\infty$. Then, almost surely when $n\to \infty$, \[
\sum_{i=1}^n \frac{g(L_{i+1})(\log T_i)^b}{T_i} =\begin{cases}
\displaystyle\frac{\EE[g(L)]}{(b+1)\EE[L]}(\log n)^{b+1} +\mathcal{O}(1) & \text{if} \; b\neq -1 \\[8pt]
\displaystyle\frac{\EE[g(L)]}{\EE[L]}\log\log n +\mathcal{O}(1) & \text{if} \; b= -1 
\end{cases}
\]
and, for all $\ell>1$, 
\[\sum_{i=1}^n \frac{g(L_{i+1})(\log T_i)^b}{T_i^{\ell}}=\mathcal{O}(1).\]
Also, for $\delta\in (0,1/2]$, almost surely as $n\to+\infty$,
\[\delta\sum_{i=1}^n g(L_{i+1})T_i^{\delta -1}=\frac{\EE[g(L)]}{\EE[L]^{1-\delta}}n^{\delta}+o(\log n).\]
\end{lemma}

\begin{proof} The law of the iterated logarithm implies that, almost surely, \[\sup_{n\to \infty} \frac{|T_n-mn|}{\sqrt{n\log n}}<\infty,\]
where $m=\EE[L]$.
Therefore, 
\begin{linenomath}\begin{align*}
\sum_{i=1}^n \frac{g(L_{i+1})(\log T_i)^b}{T_i} & = \sum_{i=1}^n \frac{g(L_{i+1})(\log (mi+\mathcal{O}(\sqrt{i\log i}))^b}{mi+\mathcal{O}(\sqrt{i\log i})} \\
&= \sum_{i=1}^n \frac{g(L_{i+1})(\log (mi))^b}{mi}\left(1+\mathcal{O}\left( \sqrt{\frac{\log i}{i}} \right)\right),
\end{align*}\end{linenomath}
where we use the $\mathcal O$-notation with the same meaning 
as in the proof of Lemma~\ref{Le: AsymptoticExpe} (see~\eqref{eq:O}). 
Therefore, almost surely as $n\to+\infty$,
\begin{linenomath}\begin{equation}\label{eq:marting1}
\sum_{i=1}^n \frac{g(L_{i+1})(\log T_i)^b}{T_i} = \sum_{i=1}^n \frac{g(L_{i+1})(\log (mi))^b}{mi}+\mathcal{O}\left( \sum_{i=1}^n \frac{g(L_{i+1})(\log (mi))^b\sqrt{\log i}}{mi^{3/2}} \right).
\end{equation}\end{linenomath}
We use martingale theory to prove that the first term in the right-hand side of Equation \eqref{eq:marting1} is almost surely equivalent to its expectation when $n\to \infty$.
We start by calculating its expectation: for all $n\geq 1$,
\[\EE\left[\sum_{i=1}^n \frac{g(L_{i+1})(\log (mi))^b}{mi} \right] = \EE[g(L)] \sum_{i=1}^n \frac{(\log(mi))^b}{mi}.\]
If $b=-1$, then, when $n\to \infty$, 
\[\sum_{i=1}^n \frac{(\log(mi))^b}{mi}
=\sum_{i=1}^n \frac{1}{mi\log(mi)}
=\frac{1}{m}\log\log n+\mathcal{O}(1).\]
If $b\neq -1$, then, when $n\to \infty$, 
\begin{linenomath}\begin{align*}
\sum_{i=1}^n \frac{(\log(mi))^b}{mi} & = \int_1^n \frac{(\log(mx))^b}{mx}\; \mathrm dx+\mathcal{O}(1)=\frac{1}{m}\int_1^{mn} \frac{(\log y)^b}{y}\; \mathrm dy +\mathcal{O}(1) \\
&= \frac{1}{m(b+1)}(\log(mn))^{b+1}+\mathcal{O}(1)
=\frac{1}{m(b+1)}(\log n)^{b+1} +\mathcal{O}(1).
\end{align*}\end{linenomath}
We thus get that, almost surely as $n\to+\infty$,
\begin{equation}\label{eq:exp_mart}
\EE\left[ \sum_{i=1}^n \frac{g(L_{i+1})(\log (mi))^b}{mi} \right]=\begin{cases}
\frac{\EE[g(L)]}{m(b+1)}(\log n)^{b+1} +\mathcal{O}(1) & \text{if} \; b\neq -1 \\[2ex]
\frac{\EE[g(L)]}{m}\log\log n +\mathcal{O}(1) & \text{if} \; b= -1.
\end{cases}
\end{equation}
For all $n\geq 1$, we let
\[M_n=\sum_{i=1}^n \frac{(g(L_{i+1})-\EE[g(L)])(\log (mi))^b}{mi}.\] 
Because the $L_i$'s are independent, $(M_n)_{n\geq 1}$ is a martingale. 
Furthermore, 
\[\Var(M_n)=\sum_{i=1}^n \frac{\Var(g(L))(\log(mi))^{2b}}{m^2i^2}
\leq \sum_{i=1}^{\infty} \frac{\Var(g(L))(\log(mi))^{2b}}{m^2i^2}<\infty,\]
which implies that $M_n$ is bounded in $L^2$ and thus 
converges almost surely to an almost-surely finite random variable. 
In other words, almost surely when $n\to +\infty$,
\begin{equation}\label{eq:first_term}
\sum_{i=1}^n \frac{g(L_{i+1})(\log(mi))^b}{mi} 
=\EE\left[\sum_{i=1}^n \frac{g(L_{i+1})(\log(mi))^b}{mi} \right]+\mathcal{O}(1)
=\begin{cases}
\frac{\EE[g(L)]}{m(b+1)}(\log n)^{b+1} +\mathcal{O}(1) & \text{if} \; b\neq -1 \\[2ex]
\frac{\EE[g(L)]}{m}\log\log n +\mathcal{O}(1) & \text{if} \; b= -1,
\end{cases}
\end{equation}
where we have used Equation~\eqref{eq:exp_mart} in the second equality.
The second term in the right-hand side of Equation~\eqref{eq:marting1} can be analysed using similar arguments, and we get that almost surely as $n\to \infty$, 
\[\sum_{i=1}^n \frac{g(L_{i+1})(\log (mi))^b\sqrt{\log i}}{mi^{3/2}}
= \sum_{i=1}^n \frac{\EE[g(L)](\log(mi))^b\sqrt{\log i}}{mi^{3/2}} +\mathcal{O}(1)
=\mathcal{O}(1).\]
Together with Equations~\eqref{eq:marting1} and~\eqref{eq:first_term}, 
this concludes the proof of the first statement of Lemma~\ref{Lemma: AsympSeries}.

The other two statements can be proved using similar arguments: 
we only detail the proof of the third statement since deriving the $\mathcal O(\log n)$ 
necessitates some additional argument when $\delta = \nicefrac12$.
Applying the law of the iterated logarithm, we get that
\[\delta\sum_{i=1}^n g(L_{i+1})T_i^{\delta -1}
= \delta \sum_{i=1}^n g(L_{i+1})(mi)^{\delta-1}
+ \mathcal O\left(\sum_{i=1}^n g(L_{i+1}) i^{\delta-\nicefrac32}\sqrt{\log i}\right),\]
almost surely when $n\to+\infty$.
The sum in this last $\mathcal O(\,\cdot\,)$ 
can be treated as above: there exists a martingale $(\Delta_n)_{n\geq 0}$, bounded in $L^2$ and thus almost surely convergent, such that
\[\sum_{i=1}^n g(L_{i+1}) i^{\delta-\nicefrac32}\sqrt{\log i}
=\sum_{i=1}^n \mathbb Eg(L) i^{\delta-\nicefrac32}\sqrt{\log i} + \Delta_n 
= \mathcal O(1),\]
almost surely when $n\to+\infty$.
This implies that, almost surely as $n\to+\infty$,
\[\delta\sum_{i=1}^n g(L_{i+1})T_i^{\delta -1}
=\delta \sum_{i=1}^n g(L_{i+1})(mi)^{\delta-1}+\mathcal O(1).\]
Note that the sum on the right-hand side is a sum of independent random variables. 
We can thus apply the strong law of large numbers (see, e.g., \cite[Th.\ IX.3.12]{petrov}), which states that, if $(D_n)_{n\geq 1}$ is a sequence of independent random variables with mean~0 and if there exists $a_n\to+\infty$ such that $\sum_{n\geq 1} a_n^{-2}\mathrm{Var}(D_n)<+\infty$, then
$a_n^{-1} \sum_{i=1}^n D_n \to 0$ almost surely when $n\to\infty$.
Since,
\[\sum_{i\geq 1} \frac{\mathrm{Var}(g(L))}{i^{2(1-\delta)}(\log i)^2}
\leq \sum_{i\geq 1} \frac{\mathrm{Var}(g(L))}{i(\log i)^2}
<+\infty,\]
this law of large numbers implies that, almost surely when $n\to+\infty$,
\[\delta\sum_{i=1}^n g(L_{i+1})(mi)^{\delta-1} 
= m^{\delta-1}\mathbb Eg(L_{i+1})\sum_{i=1}^n \delta i^{\delta-1} + o(\log n)
=m^{\delta -1} \mathbb Eg(L_{i+1}) n^\delta + o(\log n),
\]
as claimed.
\end{proof}

\begin{lemma}\label{Le: AsymptMoment} Let $f: \RR\to \RR$ be a differentiable function such that $\EE[f(L)]<\infty$ and  $\Var(f(L))<\infty$. Then, almost surely when $n\to+\infty$, 
\[
\sum_{i=1}^n \frac{\EE_{\bs{L}}[f'(F_i)]W_i}{S_i}=\begin{cases}
\frac{\EE[f(L)-f(0)]}{\EE[L]}s(n)+\mathcal O(1) & \text{if} \; \mu=\mu_1 \\[5pt]
\frac{\EE[f(L)-f(0)]}{\EE[L]^{1-\delta}}s(n)+\mathcal O(1) & \text{if} \; \mu=\mu_2 \text{ and $\delta<\nicefrac12$}\\[5pt]
\frac{\EE[f(L)-f(0)]}{\EE[L]^{1-\delta}}s(n)+\mathcal O(\log n) & \text{if} \; \mu=\mu_2 \text{ and $\delta=\nicefrac12$},
\end{cases}
\]
and, \[
\sum_{i=1}^n \frac{\EE_{\bs{L}}[f'(F_i)]W_i^2}{S_i^2} = \begin{cases}
\mathcal{O}(\log n) & \;\text{if } \mu=\mu_2 \text{ and }\delta=\frac{1}{2} \\[5pt]
\mathcal{O}(1) & \; \text{otherwise}.
\end{cases}
\]
\end{lemma}

\begin{proof} The proof is done in three steps: in Step 1.1, we look at the case when $\mu = \mu_1$ and $\beta=0$; in Step 1.2, we look at $\mu = \mu_1$ and $\beta\neq 0$; and in Step 2, we look at $\mu = \mu_2$. Without loss of generality, we assume in all three steps that $f(0) = 0$.

{\bf Step 1.1:} If $\mu=\mu_1$ and $\beta=0$, let $\tilde{\alpha}=\max\{1,\alpha \}$ and $b\in \{1,2\}$. By Lemma \ref{Le: AsymptoticExpe}, almost surely as $n\to+\infty$, 
\begin{linenomath}\begin{align*}
\sum_{i=1}^n \frac{\EE_{\bs{L}}[f'(F_i)]W_i^b}{S_i^b} 
=& \sum_{i=1}^n \frac{\alpha f(L_i) (\log T_{i-1})^{\alpha -1}W_i^{b-1}}{T_{i-1}S_i^b}
+\mathcal{O}\left(\sum_{i=1}^n \frac{\alpha L_i f(L_i)(\log T_{i-1})^{\alpha+\tilde{\alpha} -2}W_i^{b-1}}{T_{i-1}^2S_i^b} \right) \\
=& \sum_{i=1}^n  \frac{\alpha f(L_i) (\log T_{i-1})^{\alpha -1}((\log T_i)^{\alpha}-(\log T_{i-1})^{\alpha})^{b-1}}{T_{i-1}(\log T_i)^{b\alpha}}\\
& + \mathcal{O}\left( \sum_{i=1}^n \frac{\alpha L_i f(L_i)(\log T_{i-1})^{\alpha+\tilde{\alpha} -2}((\log T_i)^{\alpha}-(\log T_{i-1})^{\alpha})^{b-1}}{T_{i-1}^2(\log T_i)^{b\alpha }} \right)
\end{align*}\end{linenomath} 
Since for all $i$ large enough, $L_i/T_{i-1}\in (0,1]$,
we have 
\[\left(\frac{\log T_{i-1}}{\log T_i}\right)^{\alpha}=1+\mathcal{O}\left( \frac{L_i}{T_{i-1}\log T_{i-1}} \right),\]
and thus
\begin{equation}\label{eq:log_alpha}
(\log T_i)^{\alpha}-(\log T_{i-1})^{\alpha}=(\log T_i)^\alpha\left( 1-\left( \frac{\log T_{i-1}}{\log T_i}\right)^{\alpha}\right)=\mathcal{O}\left( \frac{L_i(\log T_i)^{\alpha}}{T_{i-1}\log T_{i-1}} \right),
\end{equation}
where the $\mathcal{O}$-notation is used as in~\eqref{eq:O}.
Therefore, almost surely as $n\to+\infty$,
\begin{linenomath}\begin{align}
\sum_{i=1}^n \frac{\EE_{\bs{L}}[f'(F_i)]W_i}{S_i} 
&= \sum_{i=1}^n \frac{\alpha f(L_i)}{T_{i-1}\log T_{i-1}} 
+\mathcal{O}\left( \sum_{i=1}^n \frac{L_if(L_i)}{T^2_{i-1}(\log T_{i-1})^2} \right)
 +\mathcal{O}\left(\sum_{i=1}^n \frac{L_i f(L_i)(\log T_{i-1})^{\tilde{\alpha} -2}}{T_{i-1}^2} \right)\notag \\
& = \sum_{i=1}^n \frac{\alpha f(L_i)}{T_{i-1}\log T_{i-1}} 
+ \mathcal{O}\left( \sum_{i=1}^n \frac{ L_i f(L_i)(\log T_{i-1})^{\tilde{\alpha} -2}}{T_{i-1}^2} \right),\label{eq:cec1}
\end{align}\end{linenomath}
and 
\begin{equation}\label{eq:cec2}
\sum_{i=1}^n \frac{\EE_{\bs{L}}[f'(F_i)]W_i^2}{S_i^2} 
= \mathcal{O}\left( \sum_{i=1}^n \frac{L_i f(L_i)}{T^2_{i-1}(\log T_{i-1})^2} \right).
\end{equation}
Applying Lemma \ref{Lemma: AsympSeries} to the two terms on the right-hand side of~\eqref{eq:cec1}, and to the righ-hand side of~\eqref{eq:cec2}, 
we get that, almost surely when $n\to +\infty$,
\[\sum_{i=1}^n \frac{\EE_{\bs{L}}[f'(F_i)]W_i}{S_i}
=\frac{\EE f(L)}{\EE L} \alpha \log \log n +\mathcal{O}(1),\]
and 
\[\sum_{i=1}^n \frac{\EE_{\bs{L}}[f'(F_i)]W_i^2}{S_i^2} 
= \mathcal{O}(1),\]
as claimed.

{\bf Step 1.2:} If $\mu=\mu_1$, $b\neq 0$, and $b\in \{1,2\}$, by Lemma \ref{Le: AsymptoticExpe},
almost surely as $n\to+\infty$,
\begin{linenomath}\begin{align}
&\sum_{i=1}^n \frac{\EE_{\bs{L}}[f'(F_i)]W_i^b}{S_i^b} \notag\\
&= \sum_{i=1}^n \frac{\alpha f(L_i) (\log T_{i-1})^{\alpha -1}\mathrm e^{\beta (\log T_{i-1})^a}W_i^{b-1}}{T_{i-1}S_i^b} + \mathcal{O}\left( \sum_{i=1}^n \frac{\alpha L_i f(L_i)(\log T_{i-1})^{\alpha-1}\mathrm e^{\beta (\log T_{i-1})^a}W_i^{b-1}}{T_{i-1}^2S_i^b} \right)\label{eq:ABn}
\end{align}\end{linenomath}
We let $A_n$ and $B_n$ denote respectively the first term and the term inside the $\mathcal O(\,\cdot\,)$ in the right-hand side of this last equation.
We have
\begin{linenomath}\begin{align}
A_n
&= \sum_{i=1}^n \frac{\alpha f(L_i) (\log T_{i-1})^{\alpha -1}\mathrm e^{\beta (\log T_{i-1})^a}\left( \mathrm e^{\beta (\log T_i)^\alpha}-\mathrm e^{\beta (\log T _{i-1})^\alpha}\right)^{b-1}}{T_{i-1}\mathrm e^{b\beta (\log T_i)^\alpha}}\notag\\
&= \sum_{i=1}^n \frac{\alpha f(L_i)(\log T_{i-1})^{\alpha -1}\mathrm e^{\beta ((\log T_{i-1})^a-(\log T_i)^{\alpha})}\left( \mathrm e^{\beta (\log T_i)^\alpha}-\mathrm e^{\beta (\log T_{i-1})^{\alpha}}\right)^{b-1}}{T_{i-1}\mathrm e^{(b-1)\beta (\log T_i)^\alpha}},\label{eq:An}
\end{align}\end{linenomath}
and
\begin{linenomath}
\begin{align}
B_n
&=\sum_{i=1}^n \frac{\alpha L_i f(L_i)(\log T_{i-1})^{\alpha-1}\mathrm e^{\beta (\log T_{i-1})^a}\left( \mathrm e^{\beta (\log T_i)^\alpha}-\mathrm e^{\beta (\log T _{i-1})^\alpha}\right)^{b-1}}{T_{i-1}^2\mathrm e^{b\beta (\log T_i)^\alpha}}\notag\\
&=\sum_{i=1}^n \frac{\alpha L_i f(L_i)(\log T_{i-1})^{\alpha-1}\mathrm e^{\beta ((\log T_{i-1})^a-(\log T_i)^{\alpha})}\left(\mathrm e^{\beta (\log T_i)^\alpha}-\mathrm e^{\beta (\log T _{i-1})^\alpha}\right)^{b-1}}{T_{i-1}^2\mathrm e^{(b-1)\beta (\log T_i)^\alpha}}\label{eq:Bn}
\end{align}\end{linenomath}
By~\eqref{eq:log_alpha},
\[\mathrm e^{\beta((\log T_{i-1})^{\alpha}-(\log T_i)^{\alpha})}
=1+\mathcal{O}\left( \frac{L_i(\log T_{i-1})^{\alpha}}{T_{i-1}} \right),\] 
and thus 
\[\mathrm e^{\beta (\log T_i)^{\alpha}}-\mathrm e^{\beta (\log T_{i-1})^{\alpha}}
=\mathcal{O}\left( \frac{L_i(\log T_{i-1})^{\alpha}\mathrm e^{\beta(\log T_i)^{\alpha}}}{T_{i-1}} \right),\] where the $\mathcal{O}$-notation is used as in~\eqref{eq:O}. 
This implies in particular that, by~\eqref{eq:An} and~\eqref{eq:Bn}, if $b=1$,
\begin{linenomath}\begin{align*}
A_n 
&= \sum_{i=1}^n \frac{\alpha f(L_i) (\log T_{i-1})^{\alpha-1}}{T_{i-1}}
+\mathcal O\left(\sum_{i=1}^n \frac{\alpha L_i f(L_i) (\log T_{i-1})^{2\alpha-1}}{T^2_{i-1}}\right)\\
B_n
&= \mathcal O\left(\sum_{i=1}^n \frac{L_i f(L_i) (\log T_{i-1})^{\alpha-1}}{T^2_{i-1}}\right),
\end{align*}\end{linenomath}
almost surely as $n\to+\infty$.
Using Lemma \ref{Lemma: AsympSeries}, we thus get that, almost surely as $n\to+\infty$,
\[A_n = \frac{\EE f(L)}{\EE L} (\log n)^\alpha +\mathcal{O}(1),\\
\quad\text{ and }\quad
B_n= \mathcal O(1).\]
Using these estimates in~\eqref{eq:ABn}, we get
\[\sum_{i=1}^n \frac{\EE_{\bs{L}}[f'(F_i)]W_i}{S_i} 
= A_n+\mathcal O(B_n) 
= \frac{\EE f(L)}{\EE L} (\log n)^\alpha +\mathcal{O}(1),\]
almost surely when $n\to+\infty$, as claimed. 
Similarly, taking $b=2$ in~\eqref{eq:An} and~\eqref{eq:Bn}, we get
\[\sum_{i=1}^n \frac{\EE_{\bs{L}}[f'(F_i)]W_i^2}{S_i^2} 
= \mathcal O(A_n)
= \mathcal{O}\left(\sum_{i=1}^n \frac{L_i  f(L_i)(\log T_{i-1})^{2\alpha -1}}{T^2_{i-1}}\right),\]
which, by Lemma \ref{Lemma: AsympSeries}, implies
\[\sum_{i=1}^n \frac{\EE_{\bs{L}}[f'(F_i)]W_i^2}{S_i^2} = \mathcal{O}(1),\]
as claimed.

{\bf Step 2:} If $\mu=\mu_2$, we proceed similarly. 
By Lemma \ref{Le: AsymptoticExpe}, for $b\in\{1,2\}$, 
\begin{linenomath}\begin{align}
\sum_{i=1}^n \frac{\EE_{\bs{L}}[f'(F_i)]W_i^b}{S_i^b}
&= \sum_{i=1}^n \frac{\gamma  \delta f(L_i) T^{\delta -1}_{i-1}\mathrm e^{\gamma T^{\delta}_{i-1}}W_i^{b-1}}{S_i^b}
+\mathcal{O}\left( \sum_{i=1}^n\frac{L_i f(L_i) T^{2(\delta -1)}_{i-1}\mathrm e^{\gamma T^{\delta}_{i-1}}W_i^{b-1}}{S_i^b} \right) \notag\\
&= \sum_{i=1}^n \frac{\gamma \delta f(L_i) T^{\delta -1}_{i-1}\mathrm e^{\gamma T^{\delta}_{i-1}}\left( \mathrm e^{\gamma T_i^{\delta}}-\mathrm e^{\gamma T_{i-1}^{\delta}}\right)^{b-1}}{\mathrm e^{b\gamma T_i^{\delta}}} 
+\mathcal{O}\left( \sum_{i=1}^n \frac{L_i f(L_i) T^{2(\delta -1)}_{i-1}\mathrm e^{\gamma T^{\delta}_{i-1}}}{\mathrm e^{b\gamma T_i^{\delta}}} \right)\label{eq:(2)}
\end{align}\end{linenomath}
almost surely as $n\to+\infty$.
Since $L_i/T_{i-1}\in [0,1]$ for $i$ large enough, and $T_{i-1}^{\delta}-T_i^{\delta}=T_{i-1}^{\delta}\big(1-(1+\frac{L_i}{T_{i-1}})^{\delta}\big)$, 
we get that 
\[T_{i-1}^{\delta}-T_i^{\delta}=\mathcal{O}( T_{i-1}^{\delta -1}L_i ).\]
Therefore, we also get
\[\mathrm e^{\gamma(T_{i-1}^{\delta}-T_i^{\delta})}
=1-\mathcal{O}(T_{i-1}^{\delta -1}L_i),\quad \text{and} \quad \mathrm e^{\gamma T_{i-1}^{\delta}}-\mathrm e^{\gamma T_i^{\delta}}=\mathcal{O}\left( T_{i-1}^{\delta -1}L_i\mathrm e^{\gamma T_i^{\delta}}\right),\]
which, together with~\eqref{eq:(2)}, imply
\[\sum_{i=1}^n \frac{\EE_{\bs{L}}[f'(F_i)]W_i}{S_i} 
= \sum_{i=1}^n f(L_i) \gamma \delta T^{\delta -1}_{i-1}
+\mathcal{O}\left( \sum_{i=1}^n L_i f(L_i) T^{2(\delta -1)}_{i-1} \right)\] 
and \[\sum_{i=1}^n \frac{\EE_{\bs{L}}[f'(F_i)]W_i^2}{S_i^2} 
= \mathcal{O}\left(\sum_{i=1}^n L_i f(L_i)T^{2\delta -2}_{i-1} \right).\]
By Lemma \ref{Lemma: AsympSeries}, almost surely as $n\to+\infty$,
\[\sum_{i=1}^n \frac{\EE_{\bs{L}}[f_{\xi}'(F_i)]W_i}{S_i}=\frac{\EE f(L)}{\EE[L]^{1-\delta}}\gamma n^{\delta} +\mathcal{O}(\log n),\]
and \[
\sum_{i=1}^n \frac{\EE_{\bs{L}}[f_{\xi}'(F_i)]W_i^2}{S_i^2} = \begin{cases}
\mathcal{O}(1) & \text{if} \; \delta<\frac{1}{2} \\[12pt]
\mathcal{O}(\log n) & \text{if} \; \delta=\frac{1}{2},
\end{cases}
\]
as claimed.
\end{proof}

\subsection{Proof of Proposition \ref{phi for S(t)}}\label{sub:proof_prop33}
Conditionally on $\bs L$, 
for all $n\geq 1$, $(F_i)_{i\geq 1}$ and $(\mathds 1_{i\prec n})_{1\leq i< n}$ are two independent sequences of independent random variables; we thus get that, for all $\xi\geq0$,
\[
\log \EE_{\bs{L}}\Big[\mathrm e^{\xi\sum_{i=1}^{n-1} F_i\mathds 1_{i\prec n}} \Big] 
 = \log \Bigg( \prod_{i=1}^{n-1} \EE_{\bs{L}}\big[\mathrm e^{\xi F_i\mathds 1_{i\prec n}} \big] \Bigg) 
= \sum_{i=1}^{n-1} \log \EE_{\bs{L}}\big[ \mathrm e^{\xi F_i\mathds 1_{i\prec n}} \big].\]
Since, conditionally on $\bs L$, 
for all $1\leq i<n$, $\mathds 1_{i\prec n}$ is Bernoulli-distributed of parameter $W_i/S_i$ (see Proposition~\ref{prop:dobr}), we get
\begin{equation}\label{eq:step1_B(t)}
\log \EE_{\bs{L}}\Big[\mathrm e^{\xi\sum_{i=1}^{n-1} F_i\mathds 1_{i\prec n}} \Big]
= \sum_{i=1}^{n-1} \log \left(1+ \frac{W_i}{S_i}\big(\EE_{\bs{L}}[\mathrm e^{\xi F_i}]-1\big)\right)
\leq \sum_{i=1}^{n-1} \frac{W_i}{S_i}\big(\EE_{\bs{L}}[\mathrm e^{\xi F_i}]-1\big) 
+R(n),
\end{equation}
where
\[|R(n)|\leq \sum_{i=1}^{n-1} \left(\frac{W_i}{S_i}\big(\EE_{\bs{L}}[\mathrm e^{\xi F_i}]-1\big)\right)^{\!2}\leq \sum_{i=1}^{n-1} \frac{W^2_i\EE_{\bs{L}}[\mathrm e^{2\xi F_i}]}{S^2_i}
,\]
because for all $x\geq 0$, $x-x^2\leq \log(1+x)\leq x$, and because $x\mapsto x^2$ is convex.
Applying Lemma \ref{Le: AsymptMoment} for $f(x)=\frac{\mathrm e^{\xi x}}{\xi}-x$, 
we get that, for all $\xi\geq 0$,
\begin{equation}\label{eq:form_with_i(t)}
\sum_{i=1}^{n-1} \frac{W_i}{S_i}\big(\EE_{\bs{L}}[\mathrm e^{\xi F_i}]-1\big) 
=\begin{cases}
\displaystyle\frac{\EE[\mathrm e^{\xi L}-1-\xi L]}{\xi\EE L}\, s(n)+\mathcal O(1) 
& \text{if } \mu=\mu_1 \\[8pt]
\displaystyle \frac{\EE[\mathrm e^{\xi L}-1-\xi L]}{\xi \EE[L]^{1-\delta}}\, s(n)
+\mathcal O(\log n)
& \text{if } \mu=\mu_2,
\end{cases}
\end{equation}
almost surely when $n\to+\infty$. 
Recall that, almost surely when $t\to+\infty$, $i(t) \sim \nicefrac t{\mathbb EL}$ 
and thus $\log(i(t)) \sim \log t$ almost surely when $t\to+\infty$. 
Therefore, almost surely when $t$ tends to infinity,
\[s(i(t)) = \begin{cases}
\alpha \log\log t + o(1) & \text{ if }\mu=\mu_1\text{ and }\beta=0,\\
s(t) + \mathcal O\big(s(t)^{1-\nicefrac1\alpha}\big)& \text{ if }\mu=\mu_1\text{ and }\beta\neq0,\\
s(t)\mathbb E[L]^{-\delta} +o(s(t))& \text{ if }\mu=\mu_2.
\end{cases}\]
Therefore, by Equation~\eqref{eq:form_with_i(t)},
\[\sum_{i=1}^{i(t)-1} \frac{W_i}{S_i}\big(\EE_{\bs{L}}[\mathrm e^{\xi F_i}]-1\big) 
=\frac{\EE[\mathrm e^{\xi L}-1-\xi L]}{\xi\EE L}\, s(t) 
+\begin{cases}
\displaystyle \mathcal O(1) 
& \text{if } \mu=\mu_1 \text{ and }\beta=0,\\
\mathcal O\big(s(t)^{1-\nicefrac1\alpha}\big) 
& \text{if } \mu=\mu_1 \text{ and }\beta\neq0,\\
o(s(t))
& \text{if } \mu=\mu_2.
\end{cases}
\]
Using again Lemma~\ref{Le: AsymptMoment} but for $f(x)=\frac{\mathrm e^{2\xi x}}{2\xi}$, we get that, almost surely when $t\to+\infty$,
\[|R(t)|\leq \begin{cases}
\mathcal O(1) 
& \text{if } \mu=\mu_1 \\
\mathcal O(\log t)
& \text{if } \mu=\mu_2,
\end{cases}
\]
Thus, in total, almost surely when $t\to+\infty$,
\[\log \EE_{\bs L}
\Big[\mathrm e^{\xi\sum_{i=1}^{i(t)-1} F_i\mathds 1_{i\prec i(t)}}\Big]
=\frac{\EE[\mathrm e^{\xi L}-1-\xi L]}{\xi\EE L} \, s(t)
+\begin{cases}
\mathcal O(1) & \text{ if } \mu=\mu_1\text{ and }\beta=0,\\
\mathcal O(s(t)^{1-\nicefrac1\alpha}) 
& \text{ if } \mu=\mu_1\text{ and }\beta\neq0,\\
o(s(t)) & \text{ if } \mu=\mu_2,
\end{cases}
\]
which concludes the proof.

\bibliographystyle{plain}   
\bibliography{Citations}

\end{document}